\documentclass[12pt]{amsart}

\usepackage[all]{xy}
\usepackage{latexsym}
\usepackage{rotating}
\usepackage{amsfonts}
\usepackage{amssymb}
\usepackage{amsmath}
\usepackage{graphics, color}
\usepackage{hyperref}

\newtheorem{theorem}{Theorem}[section] 
\newtheorem{lemma}[theorem]{Lemma}     
\newtheorem{corollary}[theorem]{Corollary}
\newtheorem{proposition}[theorem]{Proposition}
\newtheorem{remark}[theorem]{Remark}

\newtheorem{alphatheorem}{Theorem}

\newcommand{\fX}{\mathfrak{X}}

\newcommand{\C}{\mathbb{C}}
\newcommand{\hm}{\mathrm{Hom}}
\newcommand{\SL}{\mathrm{SL}}
\newcommand{\SU}{\mathrm{SU}}
\newcommand{\GL}{\mathrm{GL}}
\newcommand{\U}{\mathrm{U}}
\newcommand{\quot}{/\!\!/}
\newcommand{\bG}{\mathbf{G}}
\newcommand{\R}{\mathbb{R}}
\newcommand{\Z}{\mathbb{Z}}

\title[Fundamental Group of Moduli Spaces]{Fundamental Group of Moduli Spaces of Representations}

\date{\today}

\author[I. Biswas]{Indranil Biswas}

\address{School of Mathematics, Tata Institute of Fundamental
Research, Homi Bhabha Road, Bombay 400005, India}

\email{indranil@math.tifr.res.in}

\author[S. Lawton]{Sean Lawton}

\address{Department of Mathematical Sciences, George Mason University,
4400 University Drive,
Fairfax, Virginia  22030, USA}

\email{slawton3@gmu.edu}


 
\thanks{The first author is supported by the J. C. Bose Fellowship. The second author was partially supported by the Simons Foundation (\#245642) and the NSF-DMS (\#1309376).}

\begin{document}

\begin{abstract}
Let $\Sigma_{g,n}$ be a surface of genus $g$ with $n$ points removed, $G$ a connected
Lie group, and $\fX_{g,n}(G)$ the moduli space of representations of
$\pi_1(\Sigma_{g,n})$ into $G$.  We compute the fundamental group of $\fX_{g,n}(G)$
when $n>0$ and $G$ is a real or complex reductive algebraic group, or a compact Lie group; and when $n=0$ and
$G=\GL(m,\C)$, $\SL(m,\C)$, $\U(m)$, or $\SU(m)$.
\end{abstract}


\maketitle

\section{Introduction}

Consider $\bG$ a connected reductive algebraic group over $\C$.  We will say a Zariski 
dense subgroup $G\subset \bG$ is {\it real reductive} if $\bG(\R)_0\subset G\subset 
\bG(\R)$.  Let $\Gamma$ be a finitely generated discrete group.  Then $G$ acts on the 
analytic space $\hm(\Gamma, G)$ by conjugation.  Let $\hm(\Gamma, G)^*\,\subset\,
\hm(\Gamma, G)$ be the subspace 
with closed orbits.  Then the quotient space $\fX_{\Gamma}(G):=\hm(\Gamma,G)^*/G$
for the adjoint action is 
called the $G$-character variety of $\Gamma$.  By \cite{RiSl}, the space $\fX_{\Gamma}(G)$ is 
Hausdorff and when $G$ is real algebraic it is moreover semi-algebraic 
(and thus deformation retracts to compact space, which implies its fundamental group is 
finitely generated). Two cases of interest in this paper are when $G$ is taken to be 
$\bG$ itself and then $\fX_{\Gamma}(G)$ is the GIT quotient $\hm(\Gamma,\bG)\quot \bG$;
and when $G$ is a maximal compact 
subgroup of $\bG$, in which case $\fX_{\Gamma}(G)$ is the usual orbit space $\hm(\Gamma, G)/G$.

Moduli spaces of representations arise naturally throughout mathematics and 
physics.  In particular, when $\Gamma$ is the fundamental group of a complex manifold 
$M$, then $\fX_\Gamma(G)$ is called the Betti moduli space and is biholomorphic to the 
\v{C}ech moduli space of flat principal $G$-bundles over $M$.  This is in turn 
biholomorphic to the moduli space of flat $G$-connections on $M$, called the de 
Rham moduli space.  Lastly, there is the Dolbeault moduli space of Higgs bundles over 
$M$, which although remains diffeomorphic to the others, is equipped with a different 
complex structure which depends on $M$.  See \cite{Xia} for more about these four 
moduli spaces.

The topological study of these moduli spaces has a long history; beginning with the papers \cite{A-B, Hitchin} and \cite{Goldman-components}. More recent developments concerning their homotopy groups include \cite{BGG} for closed surface groups, and \cite{FlLaRa} for open surface groups.  

For open surface groups $\Gamma$ and connected $G$, it is easy to see that the spaces $\fX_{\Gamma}(G)$ are path-connected.  For closed surface groups $\Gamma$ and $G$ connected and compact, the components of the moduli spaces $\fX_{\Gamma}(G)$ were described in \cite{Ho-Liu-ctd-comp-I, Ho-Liu-ctd-comp-II}.  When $G$ is complex reductive and $\Gamma$ is a Riemann surface group, the components have been described in \cite{Li-surface-groups}, \cite{Lawton-Ramras}, and \cite{GaOl}.  For example, when the derived subgroup of $G$ (in either the compact or complex case) is simply connected, the moduli spaces are path-connected.  

In \cite{Lawton-Ramras}, the systematic study of the fundamental group of these moduli spaces was initiated.  In this paper, we complete part of that project.  In particular, we consider Riemann surfaces $M=\Sigma_{g,n}$ of genus $g$ with $n$ points removed.  We compute the fundamental groups of the above moduli spaces when $n>0$ for any complex reductive $\bG$ which by \cite{FlLaFree} gives the same result for compact groups $K$.  In particular, we prove:

\begin{alphatheorem}\label{thm:open}
Let $G$ be either a connected reductive algebraic group over $\C$, or a connected 
compact Lie group.  Then $\pi_1(\fX_{\pi_1\Sigma_{g,n}}(G))\,=\,\pi_1(G/[G,G])^{2g+n-1},$ 
whenever $n\,>\,0$. In particular, $\fX_{\pi_1\Sigma_{g,n}}(G)$ is simply connected if
$G$ is semisimple and $n\,>\,0$. 
\end{alphatheorem}

As a corollary, using results in \cite{CFLO}, we also determine $\pi_1(\fX_{\pi_1\Sigma_{g,n}}(G))$ when $G$ is real reductive.

After addressing moduli spaces over open Riemann surfaces, we turn to closed Riemann surfaces.

\begin{alphatheorem}\label{thm:closed}
If $G=\GL(n,\C),$ or $\U(n)$, then $\pi_1(\fX_{\pi_1\Sigma_{g,0}}(G))=\Z^{2g}.$ 
If $G=\SL(n,\C),$ or $\SU(n)$, then $\pi_1(\fX_{\pi_1\Sigma_{g,0}}(G))=0.$
\end{alphatheorem}

\begin{remark}
{\rm Theorem \ref{thm:open} extends a theorem in \cite{Lawton-Ramras}, and Theorem \ref{thm:closed} was proved in further generality in \cite{BiLaRa}.  In both cases, the methods differ from those in this 
paper.}
\end{remark}

\section*{Acknowledgments}
The second author thanks the Tata Institute of Fundamental Research for hosting him while this work was completed, and thanks Dan Ramras for comments.  We also thank an anonymous referee for helpful suggestions.

\section{Open surfaces}\label{sec:open}

Let $G$ be a connected reductive affine algebraic group over $\C$.  In this section we 
assume that $\Sigma_{g,n}$ is open, so $n\geq 1$ and $\pi_1(\Sigma_{g,n})$ is a free 
group $F_r$ of rank $r=2g+n-1$. Therefore, $\hm(\pi_1(\Sigma_{g,n}), G)\,\cong\, G^r$ 
is a smooth complex variety, and $$\fX_{g,n}(G)\,=\,G^r\quot G\, ,$$ where $G$ acts
by simultaneous conjugation given by $g\cdot (g_1,\cdots ,g_r)=(gg_1g^{-1},\cdots ,gg_rg^{-1})$.  Thus 
$\fX_{g,n}(G)$ is irreducible and normal because $G^r$ is irreducible and smooth (hence normal), and 
these properties are inherited in GIT (see \cite{Dr}, for example).  This provides some known topological properties worth pointing out that will be used in this section and in Section \ref{sec:closed}.

First, irreducible varieties are connected, and moreover any Zariski open subset is  also
connected. Normal varieties have their singular locus in  codimension at least two.
If $S$ is a connected subvariety, then there is an analytic neighborhood $U$ of $S$ such that
$S$ is a deformation retraction of $U$ (by Proposition A.5 in \cite{Hatcher}), and the complement $U\setminus S$ is connected (by \cite{Mu}).

This latter property implies that given any non-empty Zariski open subset $U\subset \fX_{g,n}$, the homomorphism
$\pi_1(U)\longrightarrow \pi_1(\fX_{g,n})$ induced by the inclusion map is surjective (see \cite{ADH}, for example).

Let $DG\,:=\,[G,G]$ be the derived subgroup of $G$, and $Z(G)$ the center of $G$.  We will denote by $T\,=\, Z_0(G)$ the connected component of $Z(G)$ containing the identity element, and $F\,=\,T\cap DG$ is a central finite subgroup.

We have an exact sequence:
$$\xymatrix{ e \ar[r]^{} & DG \ar@{^{(}->}[r] & G \ar@{->>}[r]&G/DG \ar[r]&e},$$ where
$e$ is the identity in $G$, and $G/DG\cong (\C^*)^d$ is a complex torus whose
dimension coincides with the dimension of $Z(G)$.  Moreover, $G\cong T\times_F DG$ where $F$ acts freely on $T\times DG$ by $f\cdot (t,g)=(tf,f^{-1}g)$.

If $G$ is Abelian, $\pi_1(\fX_{g,n}(G))=\pi_1(\hm(F_r,G))\cong \pi_1(G)^r,$ so we 
assume that $G$ is not Abelian.

For $r=1$, the GIT quotient $G\quot G$ deformation
retracts to  $K/K$ by \cite{FlLaFree}, and by \cite[Corollary 17]{Dal}, we have
$\pi_1(K/K)\,=\,\pi_1(K/DK)$; here $K/DK$ is the quotient for the right translation
action and not for the conjugation action.  
Therefore, we conclude that $\pi_1(G\quot G)=\pi_1(G/DG)$ since $G/DG\cong (\C^*)^d$ and $K/DK\cong (S^1)^d$ is contained in $G/DG$ (the quotient $G/DG$ is for the right translation
action). So we assume that $r\,\geq\, 2$.

\begin{remark}\label{semisimpleresult}
{\rm Also, Corollary 18 in \cite{Dal} states $K/K$ is simply-connected if and only
if $K$ is semi-simple.}
\end{remark}
 
The following discussion builds on results in \cite{Weil, JM, Goldman-symplectic}.  A representation $\rho\,\in\, \hm(F_r,G)$ is {\it irreducible} if $\rho(F_r)$ is not contained in any proper parabolic subgroup of $G$, and $\rho$ is {\it good} if it is irreducible and the centralizer of $\rho(F_r)$ in $G$ is equal to the center $Z(G)$ of $G$.

For any group $\pi_1(\Sigma_{g,n})$, the set of good representations $\hm^g(\pi_1(\Sigma_{g,n}), G)$ form a Zariski open $G$-stable subset of the polystable locus (points with closed orbits).  Moreover, when $n>0$ they are smooth as they are open subsets in a smooth manifold, and when $n=0$ they are also smooth by \cite[Proposition 37]{Sikora-CharVar}. The action of $PG$ on the set of irreducible representations is proper, and so we conclude that $$\fX^g_{g,n}(G)=\hm^g(\pi_1(\Sigma_{g,n}), G)/G\subset \fX_{g,n}(G)$$ is a smooth manifold since the $PG$-stabilizers are trivial. For open surfaces these sets are non-empty when $r\geq 2$, and for closed surfaces we need only assume $g\,\geq\,2$ for these sets to be non-empty. The dimension is computed to be $(2g+n-2)\dim G+\dim Z(G)$ in the open case, and $(2g-2)\dim G+2\dim Z(G)$ in the closed case.

\begin{lemma}
Let $G$ be a connected reductive affine algebraic group over $\C$.  Then $\hm^g(F_r,G)\to \fX^g_{g,n}(G)$ is a $PG$-bundle, and $\hm^i(F_r,G)\to \fX^i_{g,n}(G)$ is a $PG$-orbibundle.
\end{lemma}

\begin{proof}
This follows from the Slice Theorem for proper actions (see Appendix of \cite{GGK}), and the fact that $PG$ acts effectively and properly on the space of irreducible representations $\hm^i(F_r,G)$ by \cite{JM}, and locally freely by \cite{Sikora-CharVar}.  On $\hm^g(F_r,G)$ the action is free by definition.  
\end{proof}
 
A consequence of the above lemma, is that the projection mapping $\hm^g(F_r,G)\to \fX^g_{g,n}(G)$ is a Serre fibration, but by \cite{rainer} $\hm^i(F_r,G)\to \fX^i_{g,n}(G)$ is a Serre fibration if and only if it has only one orbit type (that is, exactly when the good representations are the irreducible representations).
 
\begin{theorem}\label{freetheorem}
Let $n\geq 1$ and let $G$ be either a connected reductive affine algebraic group over
$\C$, or a connected compact Lie group. Then $\pi_1(\fX_{g,n}(G))\,\cong\, \pi_1(G/DG)^r$.
\end{theorem}
 
\begin{proof}
First assume that $G$ is a connected reductive affine algebraic group over $\C$. A
continuous map will be called $\pi_1$-\textit{surjective} if the corresponding homomorphism
of fundamental groups is surjective. We will first
show that $p:\hm(F_r,G)\to \fX_{g,n}(G)$ is $\pi_1$-surjective.

Consider the following commutative diagram:  $$ \xymatrix{
\hm(F_r,G)\ar[drr]^{p}& & \\
\hm^g(F_r, G)\ar@{^{(}->}[u]^{i_g} \ar[r]_{p_g} & \fX^g_{g,n}(G)\ar@{^{(}->}[r]_{j_g}&\fX_{g,n}(G)  \\
} $$
The fibration $\hm^g(\pi_1(\Sigma_{g,n}),G)\to \fX^g_{g,n}(G)$ gives a long exact sequence in homotopy:  $$\cdots\to \pi_1(PG)\to \pi_1(\hm^g(\pi_1(\Sigma_{g,n}),G))\to \pi_1(\fX^g_{g,n}(G))\to \pi_0(PG)=0,$$  which shows that $p_g$ is $\pi_1$-surjective since $PG$ is connected.

Normality shows that the map $j_g$ is $\pi_1$-surjective.  Commutativity then implies that $p$ is $\pi_1$-surjective.

Next, the fibration $DG\to T\times DG\to T$ naturally gives a fibration $\hm(F_r,DG)\to\hm(F_r,T\times DG)\to \hm(F_r,T),$ and consequently, a fibration $\fX_{g,n}(DG)\to \fX_{g,n}(T\times DG)\to \fX_{g,n}(T),$ where $\fX_{g,n}(F)\cong F^r$ equivariantly acts freely on $\fX_{g,n}(T\times DG)\cong \fX_{g,n}(T)\times \fX_{g,n}(DG).$  This latter fibration descends to a fibration:
$$\xymatrix{\fX_{g,n}(DG) \ar@{^{(}->}[r] & \fX_{g,n}(G)
\ar@{->>}[r]&\fX_{g,n}(G/DG)},$$ since $T\times_F DG\cong G$ and $T/F\cong G/DG$ (see \cite{BiLaRa} for a generalization).

Thus, we have a long exact sequence in homotopy:
$$\cdots\to \pi_1(\fX_{g,n}(DG))\to \pi_1(\fX_{g,n}(G))\to \pi_1(\fX_{g,n}(G/DG))\to
\pi_0(\fX_{g,n}(DG))=0\, ,$$ since $\fX_{g,n}(DG)$ is connected.  But because $G/DG$ is Abelian,
we have $\fX_{g,n}(G/DG)\,\cong\, (G/DG)^r$.  So it suffices to prove that $\fX_{g,n}(DG)$ is simply-connected.

Indeed, $DG$ is semi-simple, so for any $1\leq k\leq r$ consider this commutative
diagram:  $$ \xymatrix{
DG\ar@{^{(}->}[r] \ar[d]& \hm(F_r,DG)\cong (DG)^r\ar[d]^p \\
DG\quot DG \ar[r] &\fX_{g,n}(DG)  \\
} $$
where the top map is the $k$-th factor inclusion $g\mapsto (e,...e,g,e,...e),$ and the bottom map is well-defined since the $k$-th factor inclusion is $DG$-equivariant.  By Remark \ref{semisimpleresult}, $DG\quot DG$ is simply-connected, and thus the generators of $\pi_1((DG)^r)$ all map to $0$ in $\pi_1(\fX_{g,n}(DG))$ by commutativity.  However, we have shown $p$ is $\pi_1$-surjective.  Together, this implies that $\fX_{g,n}(DG)$ is simply-connected.

In \cite{FlLaFree} is it shown that $\fX_{g,n}(K_G)$ is a deformation retraction of
$\fX_{g,n}(G)$, where $K_G$ is a maximal compact subgroup of $G$.
Therefore, the compact case follows from the complex reductive case.
\end{proof}

\begin{corollary}
Let $n\geq 1$ and let $G$ be a connected real reductive Lie group.  Then $\pi_1(\fX_{g,n}(G))\,\cong\, \pi_1(K/DK)^r$.
\end{corollary}

\begin{proof}
By \cite{CFLO}, $\fX_{g,n}(G)$ is homotopic to $\fX_{g,n}(K)$ where $K$ is a maximal compact subgroup of $G$.  By Theorem \ref{freetheorem}, we conclude that  $\pi_1(\fX_{g,n}(K))\,\cong\, \pi_1(K/DK)^r$.  
\end{proof}

\begin{remark}
We note that we cannot replace $K/DK$ in the above theorem with $G/DG$ $($as in the complex reductive case$)$; for example, $G=\mathrm{U}(p,q)$ gives a counter-example.  However, we do know in general that $K/DK$ is a geometric torus as it is a compact connected Abelian Lie group.
\end{remark}

\section{Closed surfaces}\label{sec:closed}

Let $X$ be a compact connected Riemann surface of genus $g$, with $g\, \geq\, 2$. Take
any $n\, \geq\, 2$ such that $n\,\not=\, 2$ if $g\,=\,2$. The holomorphic cotangent
bundle of $X$ will be denoted by $K_X$.

A Higgs field on a holomorphic vector bundle
$E$ on $X$ is a holomorphic section of $\text{End}(E)\otimes K_X\,=\,
E\otimes E^*\otimes K_X$. A Higgs bundle is a pair of the form $(E\, ,\theta)$,
where $E$ is a holomorphic vector bundle on $X$ and $\theta$ is a Higgs field on $E$.
A Higgs bundle $(E\, ,\theta)$ is called semistable if $\text{degree}(F)/\text{rank}(F)
\,\leq\, \text{degree}(E)/\text{rank}(E)$ for any nonzero holomorphic subbundle
$F\, \subset\, E$ with $\theta(F)\, \subset\, F\otimes K_X$.

Let ${\mathcal M}_H(X,n)$ be the moduli space of semistable Higgs bundles on $X$
of rank $n$ and degree zero. Let
$$
{\mathcal N}_H(X,n)\, \subset\, {\mathcal M}_H(X,n)
$$
be the moduli space of semistable $\text{SL}(n,{\mathbb C})$--Higgs bundles. So any
$(E\, ,\theta)\,\in\, {\mathcal M}_H(X,n)$ lies in ${\mathcal N}_H(X,n)$ if
$\bigwedge^n E\,=\, {\mathcal O}_X$ and $\text{trace}(\theta)\,=\, 0$.

\begin{proposition}\label{prop1}
The variety ${\mathcal N}_H(X,n)$ is simply connected.

The morphism ${\mathcal M}_H(X,n)\, \longrightarrow\, {\rm Pic}^0(X)$
that sends any $(E\, ,\theta)$ to the line bundle $\bigwedge^n E$ induces an isomorphism of
fundamental groups. In particular, $\pi_1({\mathcal M}_H(X,n))\,=\, {\mathbb Z}^{2g}$.
\end{proposition}

\begin{proof}
We will first show that ${\mathcal N}_H(X,n)$ is irreducible and normal. In
\cite[p. 70, Theorem 11.1]{Simpson2} it
is proved that ${\mathcal M}_H(X,n)$ is irreducible and normal (hence connected). The
same proof works for ${\mathcal N}_H(X,n)$. But normality 
of ${\mathcal N}_H(X,n)$ can be deduced from the
normality of ${\mathcal M}_H(X,n)$ as follows: let
\begin{equation}\label{phi}
\phi\, :\, {\mathcal M}_H(X,n)\, \longrightarrow\, {\mathcal M}_H(X,1)
\end{equation}
be morphism defined by $(E\, ,\theta)\,\longmapsto\, (\bigwedge^n E\, ,\text{trace}
(\theta))$. Then $\phi$ is an algebraic fiber bundle with fiber ${\mathcal N}_H(X,n)$.
Since ${\mathcal M}_H(X,1)\,=\, {\rm Pic}^0(X)\times H^0(X,\, K_X)$ is smooth, it follows
that ${\mathcal N}_H(X,n)$ is normal if and only if ${\mathcal M}_H(X,n)$ is normal.
As mentioned before, ${\mathcal M}_H(X,n)$ is normal by a theorem of Simpson.

Let ${\mathcal N}^s_X(n)$ be the moduli space of stable vector bundles on $X$ of
rank $n$ and trivial determinant. The total space of the cotangent bundle
$T^*{\mathcal N}^s_X(n)$ of the variety ${\mathcal N}^s_X(n)$ is naturally contained in 
${\mathcal N}_H(X,n)$. From the openness of the stability condition (see
\cite[p. 635, Theorem 2.8(B)]{Ma}) it follows that $T^*{\mathcal N}^s_X(n)$ is
a Zariski open subset of ${\mathcal N}_H(X,n)$.

It is known that ${\mathcal N}^s_X(n)$ is simply connected (see \cite[p. 266,
Proposition 1.2(b)]{BBGN}). Therefore, $T^*{\mathcal N}^s_X(n)$ is simply connected.
Since ${\mathcal N}_H(X,n)$ is normal and $T^*{\mathcal N}^s_X(n)$ is a simply connected Zariski
open subset of it, we conclude that ${\mathcal N}_H(X,n)$ is simply connected (see the discussion at the beginning of Section \ref{sec:open}).

Now consider the variety ${\mathcal M}_H(X,n)$. The morphism $\phi$ in \eqref{phi}
is an algebraic fiber bundle with fiber ${\mathcal N}_H(X,n)$. Since
${\mathcal N}_H(X,n)$ is connected and
simply connected, we conclude that the corresponding homomorphism
$$
\phi_{\#}\, :\, \pi_1({\mathcal M}_H(X,n))\,\longrightarrow\,
\pi_1({\mathcal M}_H(X,1))
$$
is an isomorphism. The projection
$$
q\, :\,{\mathcal M}_H(X,1)\,=\, {\rm Pic}^0(X)\times H^0(X,\, K_X)\, \longrightarrow\,
{\rm Pic}^0(X)
$$
produces an isomorphism of fundamental groups because
$H^0(X,\, K_X)$ is contractible. Therefore, the map
$$
q \circ\phi\, :\, {\mathcal M}_H(X,n)\, \longrightarrow\, {\rm Pic}^0(X)
$$
induces an isomorphism of fundamental groups.
\end{proof}

When $g,n \geq 2$, excluding $n=2=g$, Theorem \ref{thm:closed} follows from Proposition \ref{prop1} because
$\fX_{\pi_1\Sigma_{g,0}}(\GL(n,\C))$ is homeomorphic to
${\mathcal M}_H(X,n)$, and $\fX_{\pi_1\Sigma_{g,0}}(\SL(n,\C))$ is
homeomorphic to ${\mathcal N}_H(X,n)$ (see \cite[p. 31, Proposition 7.8]{Simpson2}
and \cite[p. 38, Theorem 7.18]{Simpson2}).  The $n=1$ case is trivial, and the remaining cases of Theorem \ref{thm:closed} follow from \cite{Lawton-Ramras}.


\end{document}